\documentclass{amsart}

\newtheorem{theorem}{Theorem}[section]
\newtheorem{lemma}[theorem]{Lemma}

\theoremstyle{definition}

\newtheorem*{acknowledgement}{Acknowledgement}

\usepackage{amssymb}
\usepackage{amsmath}
\usepackage{eufrak}
\usepackage{mathtools}

\theoremstyle{remark}

\numberwithin{equation}{section}

\begin{document}

\title{A VINOGRADOV-TYPE PROBLEM IN ALMOST PRIMES}
\let\thefootnote\relax\footnotetext{2010 \textit{Mathematics Subject Classification}. 11P32, 11P55.}
\let\thefootnote\relax\footnotetext{Keywords: Hardy--Littlewood circle method, Vinogradov three primes theorem, almost primes, lattices.}
\let\thefootnote\relax\footnotetext{Supported by NCN Preludium 11, 2016/21/N/ST1/02599.}

\author{Pawe\L ~ Lewulis}

\address{Institute of Mathematics, Polish Academy of Sciences, \'{S}niadeckich 8, 00-656 Warsaw, Poland \[ \]}

\email{plewulis@impan.pl}

\begin{abstract}
We generalise the Vinogradov's three primes theorem to linear equations in almost primes.  For $m\geqslant 3$ and fixed positive integers $c_1, \dots ,c_m, r_1, \dots , r_m$ we find the asymptotic formula for number of solutions in positive integers of the equation $c_1n_1 + \dots  + c_m n_m = N$, given that every $n_i$ has exactly $r_i$ prime factors. We also assume that at least three of the $r_i$ are equal to 1. The main novel idea is to introduce a variant of Vinogradov's theorem with varying coefficients and then, to create almost primes out of them in a combinatorial manner.
\end{abstract}

\maketitle

\section{Introduction}

One of the best known problems in additive combinatorics which are arleady solved is without a doubt so called weak (or ternary) Goldbach conjecture, which can be stated in the following form:

\begin{theorem}\label{1.1}
Every odd number $N$ greater than 7 can be represented as a sum of three primes.

\end{theorem}

The assertion of Theorem \ref{1.1} was proven to be correct for all sufficiently large $N$ in 1937 by Vinogradov \cite{F}. Later Chen and Wang \cite{G} gave the first effective threshold. Recently, Helfgott \cite{B,C} gave new bounds, which were strong enough to cover all remaining cases.

The proof of the ineffective version of Theorem 1.1 gives us also the precise asymptotic formula for the number of solutions of the equation $p_1 + p_2 + p_3 = N$ for $p_1,p_2,p_3$ being primes. For technical reasons it is easier to attach a weight $\log p$ to each power of a prime $p$ and deal with the sum
\\
\[ \mathcal{R}_3 (N) := \sum_{ \substack{ n_1,n_2,n_3 \\ n_1+n_2+n_3=N }  } \Lambda(n_1) \Lambda(n_2) \Lambda(n_3), \]
\\
where $\Lambda (n)$ denotes the von Mangoldt function. 

In this paper we calculate the asymptotic formula for number of solutions in almost primes of the more general equation $c_1 n_1+\dots  + c_m n_m = N$, where the $c_i$ are some fixed positive integers (linear equations and systems of linear equations in primes were studied in detail for instance in \cite{Balog,GreenTao}) and $m \geqslant 3$. Formally, we assume that the number of prime divisors of any $n_i$ is equal to $r_i$, where $(r_1, \dots ,r_m)$ is some sequence of fixed positive integers. There are also certain technical obstacles compelling us to assume that at least three of the $r_i$ are equal to $1$. We prove the following result.

\begin{theorem}\label{MAIN} Fix $m\geqslant 3$. Let $c_1, \dots , c_m$ be fixed positive integers satisfying ${(c_1, \dots ,c_m)=1}$, and let $r_1, \dots ,r_m$ be a sequence of fixed positive integers containing at least three elements being equal to $1$. Then,
\begin{multline*}
\sum_{   \substack{  n_1,\dots ,n_m \\  c_1 n_1+\dots  + c_m n_m = N \\ \forall i ~\Omega (n_i) = r_i }} 1 ~=~
\frac{1}{(m-1)!} \frac{1}{(r_1 -1)! \cdots (r_m -1)!}\frac{1}{c_1\cdots c_m} \\ \times 
\frac{  N^{m-1} }{\log ^mN} ( \log \log N)^{r_1 + \dots + r_m - m}(\mathfrak{S}_{c_1, \dots ,c_m}(N) + o(1)), 
\end{multline*}
where
\begin{multline*}
\mathfrak{S}_{c_1,\dots ,c_m}(N) := \\
\prod_{p|N} \left( 1 + \frac{\mu \left( \frac{p}{(c_1,p)} \right) \cdots {\mu \left( \frac{p}{(c_m,p)} \right)} }{\varphi \left( \frac{p}{(c_1,p)} \right) \cdots {\varphi \left( \frac{p}{(c_m,p)} \right)}}(p-1) \right)  \prod_{p \nmid N} \left( 1 - \frac{\mu \left( \frac{p}{(c_1,p)} \right) \cdots {\mu \left( \frac{p}{(c_m,p)} \right)} }{\varphi \left( \frac{p}{(c_1,p)} \right) \cdots {\varphi \left( \frac{p}{(c_m,p)} \right)}} \right),
\end{multline*}
and $\Omega (n)$ denotes the number of prime factors of $n$ counted with multiplicity. 
\end{theorem}

In Section \ref{geometry_of_numbers} we present Lemma \ref{LA.2}, proven via tools acquired from the geometry of numbers. It plays a key role in Sections \ref{major} and \ref{minor}, where we adapt the circle method to calculate the number of solutions of the equation $b_1 n_1+\dots  + b_m n_m = N$ in weighted primes, where we let the $b_i$ depend on $N$ to a certain extent (precisely, we assume that $b_i \ll N^{5/12m(2m-1)}$). Afterwards, in Section \ref{reducing} we discard the von Mangoldt weights and replace them with characteristic functions. In the last Section \ref{CuttingOff} we construct almost primes from primes in a combinatorial manner, which finishes the proof of Theorem \ref{MAIN}. In Appendix there are four technical lemmas and Lemma \ref{L3.2} which elaborates a bit on the nature of $\mathfrak{S}_{c_1,\dots , c_m}$ function. 

\subsection*{Notation}

By ${\log}$ we always denote the natural logarithm. To avoid any disturbances arising from small $N$ we assume that $N > 10\, 000$. We use the notation $\mathbf{N}=\{1,2,3,\dots \}$. We also adopt the following notions most of which are common in analytic number theory:

\begin{itemize}
\item $\varphi (n) := \left| \left( \mathbf{Z} / n \mathbf{Z} \right)^\times \right|$ denotes Euler totient function; 
\item $\tau (n)  := \sum_{d|n} 1 $ denotes the divisor function; 
\item By $(a,b)$ and $[a,b]$ we denote the greatest common divisor and the lowest common multiple, respectively;
\item For a logical formula $\phi$ we define the indicator function $\mathbf{1}_{\phi (x)}$ that equals $1$ when $\phi (x)$ is true, and $0$ otherwise;
\item By $\mathbf{R}/\mathbf{Z}$ we denote the appropriate quotient group -- in fact, in most cases we simply identify it with the interval $[0,1)$;
\item We make use of the `big $O$', the `small $o$', and the `$\ll$'  notation in a standard way. We also consider $m$ as fixed throughout the paper. 
\end{itemize}

\begin{acknowledgement} I would like to thank the anonymous referee for many helpful and very insightful comments. 
\end{acknowledgement}


\[ \]
\section{Geometry of numbers}\label{geometry_of_numbers}
Let us define $J_{b_1,\dots ,b_m}(N)$ to be the number of tuples $(n_1,\dots ,n_m)\in \mathbf{N}^m$ satisfying the equation  $b_1n_1+\dots +b_mn_m=N$ for some $b_1,\dots ,b_m\in \mathbf{N}$. 

We define a \textit{lattice} as a discrete additive subgroup of $\mathbf{R}^M$ for some $M \in \mathbf{N}$.
Let also $\mathbf{v}_1,\dots ,\mathbf{v}_K \in \mathbf{R}^M$ be linearly independent vectors, and let $K \in \mathbf{N}$.
If $\Lambda$ is a lattice and 
\[ \Lambda = \{ a_1 \mathbf{v}_1 + \dots  + a_K \mathbf{v}_K : a_1,\dots ,a_K \in \mathbf{Z} \}, \]
then the collection of points $\{\mathbf{v}_j\}_{j=1}^K$ is called a \textit{basis} of the lattice, and the set
\[ \{ t_1 \mathbf{v}_1 + \dots  + t_K \mathbf{v}_K : t_1,\dots ,t_K \in [0,1) \}\]
is called a minimal parallelogram of the lattice. Neither basis nor minimal parallelogram are unique, although the $K$-dimensional measure of this parallelogram is, and henceforth it shall be called the determinant of the lattice. We denote it as $d(\cdot)$. 

Recall the following theorem from \cite[Chapter 2.10.4, Theorem 4]{E}.
\begin{theorem}\label{TA.1}
In every lattice $L$ there exists a basis $\{\mathbf{v}_j\}_{j=1}^K$ which satisfies
\\
\[ \prod_{j=1}^K \| \mathbf{v}_j \| \ll_K d(L). \]
\end{theorem}

Our goal in this chapter is to prove the following result.

\begin{lemma}\label{LA.2}
Let $\delta \in \mathbf{R}_+$. Consider integers $1 \leqslant b_1,\dots ,b_m \leqslant N^\delta$ such that $(b_1,\dots ,b_m)=1$. Then,
\\
\[J_{b_1,\dots ,b_m}(N) = \frac{N^{m-1}}{(m-1)!} \frac{1}{b_1\cdots b_m} + O \left( N^{2\delta m (m-1) + m -2} \right). \]
\end{lemma}

\begin{proof}
Define two more lattices:
\begin{align*}
\Lambda &:= \mathbf{Z}^m, \\
 \widetilde{\Lambda} &:= \{ (n_1,\dots ,n_m) \in \mathbf{Z}^m \colon n_1+\dots +n_m=0 \}.
\end{align*}
We can transform the condition $b_1n_1 + \dots  + b_mn_m=0$ into the conjunction of $n_1+\dots +n_m=0$ and $b_1|n_1,\dots ,b_m|n_m$. Define yet another lattices:
\begin{align*}
 \Lambda^\star &:= \{ (n_1,\dots ,n_m) \in \mathbf{Z}^m \colon b_1|n_1,\dots ,b_m|n_m \}, \\
\widetilde{\Lambda}^\star &:= \{ (n_1,\dots ,n_m) \in \mathbf{Z}^m \colon  b_1|n_1,\dots ,b_m|n_m, ~n_1+\dots +n_m=0 \}. 
\end{align*}
Obviously, $\Lambda^\star \subset \Lambda$ and $\widetilde{\Lambda}^\star \subset \widetilde{\Lambda}$. 
We have
\[ \operatorname{rank}
 \begin{bmatrix}
b_1b_m & 0 & \cdots & 0 & -b_1b_m \\
0 &  b_2b_m & \cdots & 0 & -b_2b_m\\
 \vdots & \vdots & \ddots & \vdots & \vdots\\
 0 & 0 & \cdots & b_{m-1}b_m & -b_{m-1}b_m
 \end{bmatrix} = m-1.
\]
Hence, the following set of linearly independent vectors (not necessarily a basis of $\widetilde{\Lambda}^\star$):
\begin{equation}\label{vectors}
 \{ (0,\dots ,b_ib_m,\dots ,0,-b_ib_m):i=1,\dots ,m-1 \} \subset \widetilde{\Lambda}^\star ,
 \end{equation}
generates a non-degenerated parallelogram of dimension $m-1$. Its volume expressed as a square root of the modulus of the Gram's matrix can be estimated from above by the Hadamard's inequality:
\begin{equation}\label{Hadamard}
 \sqrt{ \det [ b_ib_jb_m^2(1 + \mathbf{1}_{i=j})]_{i,j=1}^{m-1}} \leqslant  2^\frac{m-1}{2} (m-1)^\frac{m-1}{4} N^{2\delta (m-1)}.
 \end{equation}
Let us denote the $\mathbf{Z}$-module generated by the set of vectors from (\ref{vectors}) by $V$.

A submodule of a finitely generated free module of rank $n$ over a principal ideal domain is also free. Moreover, the rank of this submodule cannot exceed $n$. Therefore, the lattices $V$, $\widetilde{\Lambda}^\star$, and $\widetilde{\Lambda}$, all treated as $\mathbf{Z}$-modules, have to be free. Consequently, the chain of inclusions
 \[\mathbf{Z}^{m-1} \simeq V \subset  \widetilde{\Lambda}^\star \subset \widetilde{\Lambda} \simeq \mathbf{Z}^{m-1}\]
 implies $\widetilde{\Lambda}^\star \simeq \mathbf{Z}^{m-1}.$ 

 By (\ref{Hadamard}) and Theorem \ref{TA.1} we can produce a basis $\{ \mathbf{v}_1,\dots ,\mathbf{v}_{m-1} \} \subset \widetilde{\Lambda}^\star $
satisfying
\[ \prod_{j=1}^{m-1} \| \mathbf{v}_j\| \ll d(\widetilde{\Lambda}^\star) \ll N^{2\delta (m-1)}. \]
On the other hand, the distance between each two points of $ \widetilde{\Lambda}^\star $ is not smaller than $\sqrt{2}$, which implies
$ \| \mathbf{v}_i \| \gg 1 $
for all $1 \leqslant i \leqslant m-1$. Hence,
\begin{equation}\label{hence}
\| \mathbf{v}_j \| \ll N^{2\delta (m-1)} 
\end{equation}
for $j=1,\dots ,m-1$.

Consider the minimal parallelogram
\[  P:=\{ t_1 \mathbf{v}_1 + \dots  + t_{m-1} \mathbf{v}_{m-1} \colon t_1,\dots ,t_{m-1} \in [0,1) \}. \]
We would like to know, how many points from the lattice $\widetilde{\Lambda}$ are contained in every parallelogram of the form $\mathbf{x}+P$ with $\mathbf{x} \in \widetilde{\Lambda}^\star $. To achieve this, let us consider the sequence of maps
\[ \Lambda \xrightarrow[]{~~\pi ~~}   \widetilde{\Lambda} \xrightarrow[]{~~ \rho ~~} \widetilde{\Lambda}/\widetilde{\Lambda}^\star, \]
where $\pi$ is defined as 
\[ \pi : (n_1 , \dots , n_{m-1} , n_m ) \mapsto (n_1 , \dots , n_{m-1}, -(n_1 + \dots + n_{m-1} )), \]
and $\rho$ is the module division. The kernel of $(\rho \circ \pi)$ consists of those $(n_1, \dots , n_m)$ satisfying
\begin{equation}\label{cong}
\begin{aligned}
n_1 &\equiv 0 \mod b_1 \\
&\cdots \\
n_{m-1} &\equiv 0 \mod b_{m-1} \\
n_1 + \dots + n_{m-1} &\equiv 0 \mod b_m .
\end{aligned}
\end{equation}
Let us count the number of solutions of (\ref{cong}) meeting conditions $1 \leqslant n_i \leqslant b_ib_m$ for all $1\leqslant i \leqslant m-1$, and also $n_m=0$. This problem can be rephrased a bit. We may equivalently ask about the number of solutions $(g_1, \dots , g_{m-1} )\in \mathbf{Z}^{m-1}_{b_m}$ of
\begin{equation}\label{rownowaznie}
g_1b_1 + \dots + g_{m-1}b_{m-1} \equiv 0 \mod b_m.
\end{equation}
All these solutions form a submodule of $ \mathbf{Z}^{m-1}_{b_m}$. According to Bézout's lemma, our assumption $(b_1,\dots,b_m)=1$ implies that there exist such $g_1' , \dots , g_{m-1}' \in \mathbf{Z}_{b_m}$ that 
\[ g_1'b_1 + \dots + g_{m-1}'b_{m-1} \equiv 1 \mod b_m. \]
Therefore, the whole module $\mathbf{Z}^{m-1}_{b_m}$ can be decomposed into union of $b_m$ equinumerous cosets, each defined by equation
\[ g_1'b_1 + \dots + g_{m-1}'b_{m-1} \equiv g \mod b_m, \]
with an appopriate choice of $0 \leqslant g \leqslant b_m - 1$. It implies that equation (\ref{rownowaznie}) has exactly $b_m^{m-2}$ solutions. Therefore,  (\ref{cong}) also has 
 $b_m^{m-2}$ solutions satisfying  $1 \leqslant n_i \leqslant b_ib_m$ for all $1\leqslant i \leqslant m-1$, and $n_m=0$. On the other hand, 
we know that if we replaced (\ref{cong}) by the empty set of conditions and left only the extra assumption that $n_m=0$, we would have $b_1 \cdots b_{m-1} b_m^{m-1}$ solutions. Thus, 
\begin{equation}
\# ( \Lambda / \ker ( \rho \circ \pi )  )  = b_1 \cdots b_m. 
\end{equation}
We have the isomorphism
\[   \widetilde{\Lambda}/\widetilde{\Lambda}^\star \simeq \Lambda / \ker ( \rho \circ \pi ),  \]
so $ \widetilde{\Lambda}/\widetilde{\Lambda}^\star$ has exactly  $b_1 \cdots b_m$ elements. Consequently, every parallelogram of the form $\mathbf{x}+P$ with $\mathbf{x} \in \widetilde{\Lambda}^\star $ also contains exactly $b_1\cdots b_m$ elements.

Define the following subsets
\begin{align*}
\widetilde{\Lambda}_N &:=\{ (n_1,\dots ,n_m) \in \mathbf{Z}^m \colon n_1+\dots +n_m=N \}, \\
\widetilde{\Lambda}^\star_N &:= \{ (n_1,\dots ,n_m) \in \mathbf{Z}^m \colon n_1+\dots +n_m=N, ~ b_1|n_1,\dots ,b_m|n_m  \}. 
\end{align*}
If $\mathbf{n}=(n_1,\dots ,n_m) \in \mathbf{Z}^m$ is any solution of the equation $b_1n_1+\dots +b_mn_m=N$ (which certainly exists, because $(b_1,\dots ,b_m)=1$), then we can write $\widetilde{\Lambda}_N = \mathbf{n} + \widetilde{\Lambda}$ and $   \widetilde{\Lambda}^\star_N = \mathbf{n} +   \widetilde{\Lambda}^\star $. According to this, for each $\mathbf{x} \in   \widetilde{\Lambda}^\star_N$  the parallelogram of the form $\mathbf{x} + P$ contains exactly $b_1\cdots b_m$ points from $  \widetilde{\Lambda}_N$.

The number of points of the form $   (n_1,\dots ,n_m) \in \mathbf{N}^m$, satisfying $n_1 +\dots  + n_m =N$, equals
\begin{equation}
{N - 1 \choose m-1}= \frac{N^{m-1}}{(m-1)!} + O(N^{m-2}). 
\end{equation}
For every $\mathbf{x} \in   \widetilde{\Lambda}^\star_N$ the parallelogram  $\mathbf{x} + P$ contains exactly one point from $\widetilde{\Lambda}^\star_N$. Therefore, $J_{b_1,\dots ,b_m}(N)$ is equal to the number of parallelograms of this form contained in the simplex
\[ T:= \operatorname{conv}\{ (N,0,\dots ,0),\dots ,(0,\dots ,0,N,0,\dots ,0),\dots ,(0,\dots ,0,N) \} \]
up to the number of those having non-empty intersections with $\partial T$. Let us define $R:= \operatorname{diam} P$. Observe that (\ref{hence}) implies
\begin{equation}\label{promien}
 R \leqslant \sum_{j=1}^{m-1} \| \mathbf{v}_j \| \ll N^{2\delta (m-1)}.
 \end{equation}
From $\mathbf{x} + P \subset B(\mathbf{x}, R)$ and (\ref{promien}) we obtain
\begin{equation*}
 | \{ \mathbf{x} \in \widetilde{\Lambda}^\star_N \colon \operatorname{dist}(\mathbf{x}, \partial T) \leqslant R \}| 
\leqslant  | \{ \mathbf{x} \in \Lambda \colon \operatorname{dist}(\mathbf{x}, \partial T) \leqslant R \}| \ll R^m|\partial T| \ll N^{2\delta m (m-1) + m -2}.
 \end{equation*}
We conclude that there are at most $O(N^{2\delta m (m-1) + m -2})$ parallelograms of the form $\mathbf{x}+P$ for $\mathbf{x} \in \widetilde{\Lambda}^\star_N$ having a non-empty intersection with $\partial T$. Hence, the claim follows. Also note that the main term dominates the error term for each $\delta < \frac{1}{2m(m-1)}$ which is good enough for our purposes.
\end{proof}


\[ \]
\section{Major arcs}\label{major}

Put $Q:=\log ^{B} N$ for any real $B>0$. Then, for integers $q \leqslant Q$ and $a$, such that $(a,q)=1$, we set
\[ \mathcal{M}_{a,q} = \left\{ \alpha \in \mathbf{R}/\mathbf{Z} : \| \alpha - \frac{a}{q}\|_{\mathbf{R}/\mathbf{Z}} \leqslant \frac{Q}{N} \right\} , \]
where $\| \cdot \|_{\mathbf{R}/\mathbf{Z}}$ denotes the distance to the nearest integer. We define the union of all major arcs as 
\[ \mathcal{M} := \bigcup_{q \leqslant {Q} } \bigcup_{ \substack{a=1 \\ (a,q)=1}}^{q} \mathcal{M}_{a,q}, \]
and the union of all minor arcs as $\mathfrak{m} := \mathbf{R}/ \mathbf{Z} \setminus \mathcal{M}$. The name `minor arc' may be a bit misleading since the measure of $\mathfrak{m}$ actually approaches 1 as $N \rightarrow \infty$. We put
\begin{align*}
 S_i (N, \alpha ) &:= \sum_{n \leqslant N/b_i} \Lambda (n) \, e(nb_i \alpha), \\
 S (N, \alpha ) &:= \sum_{n \leqslant N} \Lambda (n) \, e(n \alpha), \\
  u_i(y) &:=\sum_{n \leqslant N/b_i } e(nb_iy) , \\
 u(y) &:=\sum_{n \leqslant N } e(ny).
\end{align*}
Recall a useful lemma from \cite[Lemma 3.1]{Vaughan}. 

\begin{lemma}\label{Vaughan_major} There is a positive constant $C$ such that whenever $1 \leqslant a \leqslant q \leqslant Q$, $(a,q)=1$, $\alpha \in \mathcal{M}_{a,q}$ one has
\[ S (N, \alpha) = \frac{\mu (q)}{\varphi (q)} u (y) + 
   O \left( N  \exp \left(-c \sqrt{\log N} \right) \right), \]
   where $\alpha := \frac{a}{q} + y$.
\end{lemma}

It is worth mentioning that this lemma is ineffective due to its relience on the Siegel--Walfisz theorem. It also makes the whole proof of Theorem \ref{MAIN} ineffective.

Now, we prove the following result. 

\begin{theorem}\label{Valet}
Fix $c_1,\dots ,c_m \in \mathbf{N}$. Let $b_i:=c_i \eta_i$ for $i=1,\dots ,m$, where each $\eta_i$ is some positive integer with prime divisors greater than $Q$. Let us further assume that $(b_1,\dots ,b_m)=1$, and that $b_1, \dots  ,b_m \leqslant N^{\delta}$ for some $\delta \in (0,\frac{5}{12m(2m-1)})$. Then, for every $\varepsilon >0$ we have
\begin{multline*}
 \sum_{   \substack{ n_1,\dots ,n_m \\ b_1 n_1+\dots + b_m n_m = N }} \Lambda (n_1) \cdots  \Lambda (n_{m}) ~=~ \frac{1}{(m-1)!}\frac{N^{m-1}}{b_1\cdots b_m} \mathfrak{S}_{c_1,\dots ,c_m}(N) \\
+ O\left( \frac{N^{m-1}}{b_1\cdots b_m Q^{m-2-\varepsilon}} \right) +  
  \int \limits_{\mathfrak{m}} \prod_{i=1}^{m} S_i(N,\alpha)\, e(-N\alpha) \, d\alpha.
 \end{multline*}
\end{theorem}

\begin{proof}

Apply Lemma \ref{Vaughan_major} and replace $\alpha \mapsto b_i \alpha$, $y \mapsto b_i y$, and $a$ by an integer $\widetilde{a}$ satisfying $1 \leqslant \widetilde{a} \leqslant q$ and $\widetilde{a} \equiv b_i a \bmod q$. Thus, we obtain
\begin{equation}\label{3.4}
S_i (N, \alpha) = \frac{\mu \left(\frac{q}{(b_i,q)} \right)}{\varphi \left( \frac{q}{(b_i,q)} \right) }u_i (y) + 
    O \left( \frac{N}{b_i } \exp \left( -C \sqrt{ \log N} \right) \right)
\end{equation}
for some positive constant $C$.
We are ready to estimate the contribution of a single major arc to the integral:
\begin{multline}\label{3.5}
\int \limits_{\mathcal{M}_{a,q}} \prod_{i=1}^m S_i (N, \alpha) \, e(-N \alpha) \, d \alpha =  \prod_{i=1}^m  \frac{\mu \left(\frac{q}{(b_i,q)} \right)}{\varphi \left( \frac{q}{(b_i,q)} \right) } \, e\left( - \frac{aN}{q} \right)~ \times \\
\int \limits_{-\frac{Q}{N}}^{\frac{Q}{N}} \prod_{i=1}^m u_i (y) \, e(-Ny) \, dy + O \left(  \sum_{ \substack{ \omega \in \{ 0,1 \}^m \\ \omega \not= (1,\dots ,1) }}  \int \limits_{-\frac{Q}{N}}^{\frac{Q}{N}} f_{\omega_1}(y)\cdots f_{\omega_m}(y) e(-Ny)\, dy\right),
\end{multline}
where 
\begin{equation}\label{3.6}
  f_{\omega_j}(y):=\begin{cases}
   u_i(y) \times \mu \left(\frac{q}{(b_i,q)} \right)/\varphi \left( \frac{q}{(b_i,q)} \right)   & \text{if $\omega_j = 1$}\\
     \frac{N}{b_i } \exp \left( -C \sqrt{ \log N} \right) & \text{if $\omega_j = 0$ }
  \end{cases}.
\end{equation}

We have the obvious inequality $|u_i(y)| \leqslant \frac{N}{b_i}$. Thus, if at least one coordinate of $\omega$ is equal to 0, then for such $\omega$ we get
\\
\begin{equation}\label{3.7} \int \limits_{-\frac{Q}{N}}^{\frac{Q}{N}} f_{\omega_1}(y)\cdots f_{\omega_m}(y) \, e(-Ny)\, dy \ll
\frac{1}{b_1\cdots b_m} \frac{N^{m-1}}{ \exp( (C-\varepsilon)\sqrt{\log N} ) }.
\end{equation}

Summing over every possible $a,q$ one gets
\\
\begin{multline}\label{3.8}
\int \limits_{\mathcal{M}} \prod_{i=1}^m S_i (N, \alpha) \, e(-N \alpha) \, d \alpha = 
 \sum_{q \leqslant Q} 
\prod_{i=1}^m  \frac{\mu \left(\frac{q}{(b_i,q)} \right)}{\varphi \left( \frac{q}{(b_i,q)} \right) } \sum_{ \substack{ a\leqslant q \\ (a,q)=1 }}  e\left( - \frac{aN}{q} \right)  \int \limits_{-\frac{Q}{N}}^{\frac{Q}{N}} \prod_{i=1}^m u_i (y) \, e(-Ny) \, dy
\\
+ O \left(  \frac{N^{m-1}}{ b_1\cdots b_m \exp \left( (C-3\varepsilon)\sqrt{\log N} \right) } \right) .
\end{multline}
Put $c_q(N) :=  \sum_{ \substack{ a\leqslant q \\ (a,q)=1 }}  e\left( - \frac{aN}{q} \right)$. Due to multiplicativity of Ramanujan's sums we conclude that
\begin{multline}\label{3.9}
\sum_{q \leqslant Q} 
\prod_{i=1}^m  \frac{\mu \left(\frac{q}{(b_i,q)} \right)}{\varphi \left( \frac{q}{(b_i,q)} \right) } c_q(N) =   \prod_{p|N} \left( 1 + \prod_{i=1}^m  \frac{\mu \left(\frac{p}{(c_i,p)} \right)}{\varphi \left( \frac{p}{(c_i,p)} \right) } (p-1) \right)
 \prod_{p \nmid N} \left( 1 - \prod_{i=1}^m  \frac{\mu \left(\frac{p}{(c_i,p)} \right)}{\varphi \left( \frac{p}{(c_i,p)} \right) } \right)
\\
+  O \left( \sum_{q>Q} \prod_{i=1}^m \frac{\mu \left(\frac{q}{(c_i,q)} \right)}{\varphi \left( \frac{q}{(c_i,q)} \right) } c_q(N) \right)= \mathfrak{S}_{c_1,\dots ,c_m}(N) + O \left( \frac{1}{Q^{m-2-\varepsilon} } \right).
\end{multline}
We could replace $b_i$ by $c_i$ because $(b_i,q)=(c_i,q)$ for each $q\leqslant Q$ and each $1 \leqslant i \leqslant m$. 
Moreover,

\[  |\mathfrak{S}_{c_1,\dots ,c_m}(N)| < \prod_{p} \left( 1 + \frac{1}{\varphi(p)^{m-1}} \right) \ll 1. \]

We are left with the integral on the right hand side of (\ref{3.8}) to estimate. Let us consider the following subsets of $\mathbf{R}/\mathbf{Z}$:

\[ J_k^{(j)} =  \left[ \frac{k}{b_j} - \frac{1}{b_jN^{1/3}},  \frac{k}{b_j} + \frac{1}{b_jN^{1/3}} \right], \]
\[  I_k^{(j)} =  \left[ \frac{k}{b_j} - \frac{1}{b_jN^{1/2}},  \frac{k}{b_j} + \frac{1}{b_jN^{1/2}} \right], \]
for $j=1,\dots ,m$ and $k=0,\dots ,b_i-1$. The distance between two distinct fractions of the form $k/b_j$ satisfies
\\
\begin{equation}\label{3.10}
\left| \frac{k_1}{b_{j_1}} - \frac{k_2}{b_{j_2}} \right|  \geqslant \frac{1}{b_{j_1}b_{j_2}} \geqslant \max \left\{ \frac{1}{b_{j_1}N^{\delta}},  \frac{1}{b_{j_2}N^{\delta}}  \right\} \geqslant \frac{1}{b_{j_1}N^{1/3}} +  \frac{1}{b_{j_2}N^{1/3}} 
\end{equation}
\\
for $N>2^{\frac{3}{1-3\delta}}$ and arbitrary $k_1,k_2 \in \mathbf{Z}$. Consequently, we can assume that $N$ is so large that each pair of intervals of the form $J_k^{(j)}$ centered in different points $k/b_j$ have empty intersection.

From $I_k^{(j)} \subset J_k^{(j)}$ one can decompose\footnote{The idea regarding the $J_k^{j}$ and the $I_k^{(j)}$ intervals is heavily inspired by the comment of user Tal H from \texttt{mathoverflow.net}. }
\begin{multline}\label{3.11}
J_{b_1,\dots ,b_m}(N) = \int \limits_{\mathbf{R}/\mathbf{Z}} \prod_{i=1}^m u_i (y) \, e(-Ny) \, dy =: \int \limits_{\mathbf{R}/\mathbf{Z}}  \\
= \int \limits_{-\frac{Q}{N}}^{\frac{Q}{N}}
+ ~ O\left( \sum_{j=1}^{m} \sum_{k=1}^{b_j -1} \left| ~ \int \limits_{I_k^{(j)}} + 
\int \limits_{J_k^{(j)} \setminus I_k^{(j)}} \right| \right) + \int \limits_{\frac{Q}{N}}^{\frac{1}{bN^{1/3}}} + \int \limits^{-\frac{Q}{N}}_{-\frac{1}{bN^{1/3}}}+ \int \limits_{\mathcal{S}}, 
\end{multline}
where $b:=\min \{ b_1,\dots ,b_m \}$ and $\mathcal{S}$ denotes the set $\mathbf{R}/\mathbf{Z}\setminus \bigcup_{j=1}^m \bigcup_{k=0}^{b_j -1} J_k^{(j)}$. Now, we recall the basic Dirichlet kernel estimation: 
\[ u_i (y) \leqslant \frac{1}{1 - e(b_i y)} \ll \frac{1}{\| b_i y \|_{\mathbf{R}/\mathbf{Z}}}. \]
Applying it, we obtain
\begin{equation}\label{3.12}
  \left|u_i (y) \right| \ll
\begin{cases}
	N^{1/3} & \text{if $\forall_k ~y \not \in J_k^{(i)}$}\\
	N^{1/2} & \text{if $\forall_k ~y \not \in I_k^{(i)}$}\\
	\frac{1}{|b_i y|} & \text{if $y \in \left[-\frac{1}{bN^{1/3}}, \frac{1}{bN^{1/3}} \right]\setminus \{0 \} $}\\
	N & \text{always}
  \end{cases}.
\end{equation}
The assumption $(b_1,\dots ,b_m)=1$ implies that for each $y_0 \in (0,1)$ at most $m-1$ elements from the set $ \{ b_1y_0, \dots , b_my_0 \}$ can be integers. Hence, for every $y \in {\mathbf{R}/\mathbf{Z}} \setminus \left[-\frac{1}{bN^{1/3}}, \frac{1}{bN^{1/3}} \right] $ there exists such an index $j \in \{ 1, \dots , m \}$ that none of the $J_k^{(j)}$ contains $y$. Thus,

\begin{equation}\label{3.13}
  \left| \prod_{i=1}^m u_i (y) \right| \ll
\begin{cases}
	N^{m-1+\frac{1}{3}} & \text{if $y \in \bigcup_{j=1}^m \bigcup_{k=1}^{b_i -1} I_k^{(j)}$}\\
	N^{m-2 + \frac{1}{2} + \frac{1}{3}} & \text{if $y \in \bigcup_{j=1}^m \bigcup_{k=1}^{b_i -1} ( J_k^{(j)} \setminus I_k^{(j)}) $}\\
	\frac{1}{b_1\cdots b_m} \frac{1}{|y|^m} & \text{if $y \in \left[-\frac{1}{bN^{1/3}}, \frac{1}{bN^{1/3}} \right]\setminus \{0 \} $}\\
	N^{\frac{m}{3}} & \text{if $y \in \mathcal{S}$}\\
  \end{cases}.
\end{equation}

From (\ref{3.13}) we calculate
\begin{equation}\label{3.14}
\begin{gathered}
\int \limits_{\mathcal{S}} \ll N^{\frac{m}{3}},\\
 \left( \int \limits_{\frac{Q}{N}}^{\frac{1}{bN^{1/3}}} + \int \limits^{-\frac{Q}{N}}_{-\frac{1}{bN^{1/3}}} \right) \ll \frac{1}{b_1\cdots b_m} \int \limits_{\frac{Q}{N}}^{\frac{1}{bN^{1/3}}} \frac{dy}{y^m} \ll  \frac{1}{b_1\cdots b_m} \left( \frac{N}{Q} \right)^{m-1}, \\
 \sum_{j=1}^{m} \sum_{k=1}^{b_j -1} \left| ~ \int \limits_{I_k^{(j)}} + 
\int \limits_{J_k^{(j)} \setminus I_k^{(j)}} \right| \ll  \sum_{j=1}^{m} \sum_{k=1}^{b_j -1} \left( N^{m-1+\frac{1}{3}}\frac{1}{b_j N^{1/2}} + N^{m-2 + \frac{1}{2} +\frac{1}{3}} \frac{1}{b_j N^{1/3}} \right) \ll N^{(1+\delta)m-\frac{7}{6}}. 
\end{gathered}
\end{equation}

Combining (\ref{3.11})--(\ref{3.14}), and recalling that $\delta < \frac{5}{12m(2m-1)} \leqslant \frac{1}{12m}$, we conclude that
\begin{equation}\label{3.15}
 \int \limits_{-\frac{Q}{N}}^{\frac{Q}{N}} \prod_{i=1}^m u_i (y) \, e(-Ny) \, dy = J_{b_1,\dots ,b_m}(N) + O\left( 
 \frac{1}{b_1\cdots b_m}  \left( \frac{N}{Q} \right)^{m-1} \right).
\end{equation}

From (\ref{3.8}), (\ref{3.9}), (\ref{3.15}), and Lemma \ref{LA.2} we get
\begin{equation}\label{3.16}
\int \limits_{\mathcal{M}} \prod_{i=1}^m S_i (N, \alpha) \, e(-N \alpha) \, d \alpha =  \frac{1}{(m-1)!}\frac{N^{m-1}}{b_1\cdots b_m}
\mathfrak{S}_{c_1,\dots ,c_m}(N) + O\left( \frac{1}{b_1\cdots b_m} \frac{N^{m-1}}{Q^{m-2-\varepsilon}} \right).
\end{equation}
\end{proof}


\[ \]
\section{Minor arcs}\label{minor}

In this section we estimate the contribution of the integral 
\\
\begin{equation}\label{4.1}
 \int \limits_\mathfrak{m} \prod_{i=1}^{m} S_i(N,\alpha) \, e(-N\alpha) \, d\alpha.
\end{equation}
\\
Usually, some variations of the Vinogradov's lemma (Lemma \ref{4.4} in our case) are exploited to establish results of this kind. Firstly, let us recall yet another lemma.

\begin{lemma}[R.~C.~Vaughan]\label{L4.1} Let $X,Y,\alpha \in \mathbf{R}$ where $X,Y\geqslant 1$. Assume that $|\alpha - \frac{a}{q}| \leqslant \frac{1}{q^2}$ for some $a,q \in \mathbf{N}$ such that $(a,q)=1$. Thus,
\[ \sum_{n \leqslant X} \min \left\{ \frac{XY}{n} , \frac{1}{\| n \alpha \|_{\mathbf{R}/\mathbf{Z}}} \right\} \ll \left( \frac{XY}{q}+ X + q \right)\log(2Xq). \]
\end{lemma}

\begin{proof}
See \cite[Lemma 2.2]{Vaughan}.
\end{proof}

We need the following variation of the Vaughan's identity \cite{Vaughan2}.

\begin{lemma}\label{L4.2} For any real $x, \alpha$, and for $2 \leqslant U,V \leqslant x$ we have the identity
\[ S(x,\alpha) = S_{I,1} - S_{I,2} - S_{II} + S_0, \]
where
\begin{align*}
S_{I,1} &= \sum_{d \leqslant U} \mu (d) \sum_{n \leqslant \frac{x}{d}} ( \log n ) \, e(n d \alpha), \\
S_{I,2} &= \sum_{d \leqslant V} \Lambda (d) \sum_{\delta \leqslant U} \mu (\delta ) \sum_{n \leqslant \frac{x}{d \delta}} e(nd\delta \alpha ), \\
S_{II} &= \sum_{d > U}  \bigg( \sum_{\substack{ \delta \leqslant U \\ \delta | d }} \mu (\delta ) \bigg)  \sum_{\substack{n>V \\ nd \leqslant x }} \Lambda (n) \, e(n d \alpha), \\
S_0 &= \sum_{n \leqslant V} \Lambda (n) \, e(n \alpha). 
\end{align*}
\end{lemma}

We aim to perform our minor arc estimation by applying the following key lemma. 

\begin{lemma}\label{L4.4}  Assume that $1 \leqslant b_i \leqslant N$ and $b_i \in \mathbf{N}$ for each $i=1,\dots , m$. Let $|\alpha - \frac{a}{q}|\leqslant \frac{1}{q^2}$ for some $a,q\in \mathbf{N}$ such that $(a,q)=1$ and $q\leqslant N$. Hence,
\[ S_i(N, \alpha) \ll \log^4 N  \left(   \frac{N}{\sqrt{q} } +  N^\frac{4}{5}b_i^{\frac{1}{5}} + \sqrt{Nq} \right). \]
\end{lemma}
\begin{proof}
We follow the reasoning described in \cite[Chapter 3]{Vaughan}. Let us apply Lemma \ref{4.2} with $x:=\frac{N}{b_i}$ and $b_i \alpha$ instead of $\alpha$. 

The inner sum in $S_{I,1}$ equals
\[
\sum_{n \leqslant \frac{N}{d b_i }} e(ndb_i \alpha ) \int \limits_1^n \frac{dy}{y} = \int \limits_1^{\frac{N}{d b_i }} \sum_{n \leqslant \frac{N}{d b_i }} e(nd b_i \alpha ) \mathbf{1}_{y < n} \frac{dy}{y} =
\int \limits_1^{\frac{N}{d b_i}} \sum_{y < n \leqslant \frac{N}{d b_i }} e(nd b_i \alpha ) \frac{dy}{y},
\]
which gives 

\begin{equation}\label{4.2}
 S_{I,1} \ll \log N \sum_{d \leqslant U} \min \left\{ \frac{N}{d b_i} , \frac{1}{\| d b_i \alpha  \|_{\mathbf{R}/\mathbf{Z}}} \right\} =
  \log N \sum_{\substack{ d \leqslant Ub_i \\ b_i|d }} \min \left\{ \frac{N}{d} , \frac{1}{\|  d \alpha  \|_{\mathbf{R}/\mathbf{Z}}} \right\}.
\end{equation}
Furthermore, one can see that

\begin{equation}\label{4.3}
\begin{gathered}
S_{I,2} = \sum_{\delta_1 \leqslant U, \delta_2 \leqslant V} \sum_{d \leqslant \frac{N}{\delta_1 \delta_2 b_i}} \mu (\delta_1 ) \Lambda (\delta_2) \, e(\delta_1 \delta_2 d b_i \alpha)   \\
=  \sum_{\delta_1 \leqslant U, \delta_2 \leqslant V} \sum_{d \leqslant \frac{N}{\delta_1 \delta_2 b_i}} \sum_{\substack{ n \leqslant UV \\ \delta_1 \delta_2 = n}} \mu (\delta_1 ) \Lambda (\delta_2) e(n d b_i \alpha)  
~=  \sum_{ \substack {\delta_1 \leqslant U, \delta_2 \leqslant V, d, n \leqslant UV \\ \delta_1 \delta_2 = n, ~dn \leqslant 
\frac{N}{b_i}  }} \mu (\delta_1 ) \Lambda (\delta_2) e(n d b_i \alpha)   \\
 = \sum_{n \leqslant UV }  \left( \sum_{ \substack {\delta_1 \leqslant U, \delta_2 \leqslant V \\ \delta_1 \delta_2 = n }} \mu(\delta_1 ) \Lambda (\delta_2 ) \right) \sum_{d \leqslant \frac{N}{nb_i}} e(n d b_i \alpha) ~\ll~
\log (UV) \sum_{n \leqslant UV }  \min \left\{ \frac{N}{nb_i} , \frac{1}{ \| nb_i \alpha \|_{\mathbf{R}/\mathbf{Z}}  } \right\}  \\
=~  \log (UV) \sum_{\substack{ d \leqslant UVb_i \\ b_i | d }}  \min \left\{ \frac{N}{d} , \frac{1}{ \|  d \alpha \|_{\mathbf{R}/\mathbf{Z}} } \right\}. 
\end{gathered}
\end{equation}
To estimate the value of $S_{II}$ let us define the set
 \[\mathcal{Y} := \{ U, 2U, 4U, \dots  , 2^kU : 2^k UV < N/b_i \leqslant 2^{k+1} UV \}. \]
Thus,
\[ S_{II} = \sum_{Z \in \mathcal{Y} } \mathcal{S} (Z), \]
where
\[ \mathcal{S} (Z) := \sum_{Z < d \leqslant 2Z}    \left( \sum_{\substack{ \delta \leqslant U \\ \delta | d }} \mu (\delta ) \right) \sum_{V < n \leqslant \frac{N}{db_i}}\Lambda (n) e(nd b_i \alpha). \]
By the Cauchy--Schwarz inequality
\[ |\mathcal{S} (Z)|^2 \leqslant \left( \sum_{Z < e \leqslant 2Z} \tau (e)^2 \right) \sum_{Z < d \leqslant 2Z} \left|  \sum_{V < n \leqslant \frac{N}{db_i}}\Lambda (n) e(nd b_i \alpha) \right|^2. \]
From Lemma \ref{tau_sq} we get
\begin{equation}\label{4.4}
\begin{aligned}
 |\mathcal{S} (Z)|^2 &\ll Z \log^3 N    \sum_{V < n_1, n_2 \leqslant \frac{N}{Zb_i}} \Lambda (n_1) \Lambda (n_2)
 \sum_{Z < d \leqslant 2Z}  e((n_1 - n_2)d b_i \alpha) \\
&\ll   Z \log^5 N    \sum_{ n_1, n_2 \leqslant \frac{N}{Zb_i}} \min  \left\{  Z ,  \frac{1}{ \| (n_1-n_2)b_i \alpha \|_{\mathbf{R}/\mathbf{Z}} }   \right\}  \\
&=    Z \log^5 N  \sum_{ \substack{ n \leqslant \frac{N}{Z} \\ b_i|n}}   \sum_{ \substack{ -n \leqslant d \leqslant \frac{N}{Z} - n \\ b_i|d}}    
\min  \left\{  Z ,  \frac{1}{ \|  d \alpha \|_{\mathbf{R}/\mathbf{Z}} }   \right\} \\
&\ll \frac{N}{b_i} \log^5 N  \left( Z+  \sum_{ \substack{  d \leqslant \frac{N}{Z} \\ b_i|d}}    
\min  \left\{  \frac{N}{d} ,  \frac{1}{ \|  d \alpha \|_{\mathbf{R}/\mathbf{Z}} }   \right\} \right),
\end{aligned}
\end{equation}

Let us combine (\ref{4.2})--(\ref{4.4}) and put $U,V:= \left( \frac{N}{b_i} \right)^\frac{2}{5}$. Assume that $|\alpha - \frac{a}{q}|\leqslant \frac{1}{q^2}$ for some $a,q\in \mathbf{N}$ such that $(a,q)=1$, $q\leqslant N$. Applying Lemma \ref{L4.1} one gets
\begin{equation}\label{sumki_minor}
\begin{aligned}
 \sum_{d \leqslant Ub_i  } \min \left\{ \frac{N}{d} , \frac{1}{\| d \alpha  \|_{\mathbf{R}/\mathbf{Z}}} \right\} &\ll
  \left(\frac{N}{q} + N^\frac{4}{5} b_i^\frac{1}{5} + q \right) \log N, \\
   \sum_{n \leqslant UVb_i }  \min \left\{ \frac{N}{n} , \frac{1}{ \| n \alpha \|_{\mathbf{R}/\mathbf{Z}} } \right\} &\ll \left(\frac{N}{q} + N^\frac{4}{5} b_i^\frac{1}{5} + q \right) \log N,  \\
 Z+  \sum_{ \substack{  d \leqslant \frac{N}{Z}}}\min  \left\{  \frac{N}{d} ,  \frac{1}{ \| \, d \alpha \|_{\mathbf{R}/\mathbf{Z}} }   \right\}  &\ll \left(  \frac{N}{q} + \frac{N}{Z} + q + Z \right) \log N.
\end{aligned}
\end{equation}
Note that the restricions $b_i|d$ under summands from (\ref{4.2}), (\ref{4.3}), and (\ref{4.4}) were simply discarded, so there might be some room for improvement here. Now, (\ref{4.2})--(\ref{sumki_minor}) imply
\begin{equation}\label{4.5_wniosek}
\begin{aligned}
S_{I,1} &\ll  \left(\frac{N}{q} + N^\frac{4}{5} b_i^\frac{1}{5} + q \right) \log^2 N, \\
S_{I,2} &\ll \left(\frac{N}{q} + N^\frac{4}{5} b_i^\frac{1}{5} + q \right) \log^2 N.
\end{aligned}
\end{equation}
 Since $| \mathcal{Y} | \ll \log N$ and 
\[ |\mathcal{S}(Z)|^2 \ll \frac{N}{b_i} \log^6 N \left(  \frac{N}{q} + \frac{N}{Z} + q + Z \right),\]
one also gets
 \begin{equation}\label{S_2}
 S_{II} \ll  \frac{ \log^3 N}{\sqrt{b_i}}  \sum_{Z \in \mathcal{Y}}  \left(   \frac{N}{\sqrt{q} } +  \frac{N}{\sqrt{Z} } + \sqrt{Nq} + \sqrt{ NZ} \right) \ll   \frac{ \log^4 N }{\sqrt{b_i}}  \left(   \frac{N}{\sqrt{q} } +  N^\frac{4}{5} b_i^\frac{1}{5} + \sqrt{Nq} \right).\
 \end{equation}
 
 The claim follows from (\ref{4.5_wniosek})--(\ref{S_2}) merged with the inequality $S_0 \ll N^{\frac{2}{5}}$.
 \end{proof}

Recall that $Q:= \log^B N$, and that $\alpha \in \mathfrak{m}$ means that $\alpha$ does not satisfy the approximation
\[ \left| \alpha - \frac{a}{q} \right|  \leqslant \frac{Q}{N}  \]
for any $1 \leqslant q \leqslant Q$ and any integer $a$ coprime to $q$. We present the following lemma.

\begin{lemma}\label{L4.5}
Fix real $B, \varepsilon >0$. Let $b_i \leqslant N^{1 - \varepsilon}$ and $b_i \in \mathbf{N}$ for each $i=1,\dots , m$. Hence, for any $\alpha \in \mathfrak{m}$ we have
\[ S_i (N, \alpha) \ll \frac{N}{\log^{\frac{B}{2} -4} N}.\]
\end{lemma}

\begin{proof} By the Dirichlet's approximation theorem, for any $\alpha \in \mathbf{R}$ there exists a positive integer $q \leqslant \frac{N}{Q}$ such that $|\alpha - \frac{a}{q}|\leqslant \frac{Q}{qN} \leqslant \frac{1}{q^2}$ for some $a \in \mathbf{N}$ satisfying the condition $(a,q)=1$. On the other hand, it has to be $q>Q$, because otherwise $\alpha \in \mathcal{M}$ would follow. Hence, by Lemma \ref{L4.4} the claim is true.
\end{proof}

Now, we summarize the section by proving the following result.

\begin{lemma}\label{L4.6}
Fix a real $B>0$. Under the assumptions of Theorem \ref{Valet} and the extra assumption $\eta_1, \eta_2, \eta_3 =1$ we have
\[  \int \limits_\mathfrak{m} \prod_{i=1}^{m} S_i(N,\alpha) e(-N\alpha) \, d\alpha \ll \frac{1}{b_1\cdots b_m}\frac{N^{m-1}}{\log^{\frac{B}{2}-3}N}.\]
\end{lemma}

\begin{proof}
We have
\begin{equation}
\begin{gathered}
\int \limits_\mathfrak{m} \prod_{i=1}^{m} S_i(N,\alpha) e(-N\alpha) \, d\alpha \leqslant \prod_{i=3}^m \max_{\alpha \in \mathfrak{m}} |S_i (N,\alpha)| \int \limits_{\mathbf{R}/\mathbf{Z}}| S_1 (N, \alpha) S_2 (N,\alpha)| \, d \alpha\\
 \leqslant  \frac{N^{m-3}}{b_4 \cdots b_m} \left( \max_{\alpha \in \mathfrak{m}} |S_3 (N,\alpha)| \right) \left( \, \int \limits_{\mathbf{R}/\mathbf{Z}}| S_1 (N, \alpha)|^2 \, d \alpha \right)^{1/2} \left(
\, \int \limits_{\mathbf{R}/\mathbf{Z}}| S_2 (N,\alpha)|^2 \, d \alpha \right)^{1/2} \\
\ll \frac{N^{m-3}}{b_4 \cdots b_m} \frac{N}{\log^{\frac{B}{2} -4} N} \frac{N}{\sqrt{b_1b_2}} \sqrt{ \log (N/b_1) \log (N/b_2)} \ll
\frac{N^{m-1}}{b_1 \cdots b_m} \frac{1}{\log^{\frac{B}{2} - 3} N},
\end{gathered}
\end{equation}
which follows from the Cauchy--Schwarz inequality, the Plancherel identity, the basic estimation $\sum_{n\leqslant x} \Lambda (n)^2 \ll x \log x$ valid for $x\geqslant 1$, Lemma \ref{L4.5} and the fact that $b_1,b_2,b_3$ are considered fixed.
\end{proof}
We combine the results from the last two sections and establish the following result.

\begin{theorem}\label{T4.7} Fix a real $B>6$ and $c_1,\dots ,c_m \in \mathbf{N}$. Let $b_i:=c_i \eta_i$ for $i=1,\dots ,m$, where each $\eta_i$ is some positive integer with prime divisors greater than $Q$. Let us further assume that $(b_1,\dots ,b_m)=1$, that $b_1, \dots  ,b_m \leqslant N^{\delta}$ for some $\delta \in (0,\frac{5}{12m(2m-1)})$, and that $\eta_1,\eta_2,\eta_3=1$. Hence, 
\[ \sum_{   \substack{ n_1, \dots ,n_m \leqslant N \\ b_1 n_1+ \dots + b_m n_m = N }} \Lambda (n_1) \cdots \Lambda (n_{m}) = \frac{1}{(m-1)!}\frac{N^{m-1}}{b_1 \cdots b_m} \mathfrak{S}_{c_1, \dots ,c_m}(N) + O\left( \frac{N^{m-1}}{b_1\cdots b_m \log^A N} \right), \]
where $A=(B-6)/2$.
\end{theorem}

\begin{proof}
Follows from Theorem \ref{Valet} and Lemma \ref{L4.6} upon taking $\varepsilon = \frac{1}{2}$.
\end{proof}

\[ \]
\section{Reducing the logarithmic weights}\label{reducing}

Firstly, for each $k>1$ we would like to discard the contribution of the numbers of the form $p^k$ from the expression appearing in the left-hand side of the equation from the statement of Theorem \ref{T4.7}. For the sake of convenience, let us introduce the following notation:
\begin{equation}
\begin{aligned}
\sideset{}{^\flat}\sum_{n_1, \dots, n_m}  := \sum_{ \substack { n_1, \dots, n_m \text{ being prime}\\ b_1 n_1 + \dots + b_mn_m=N  }},
\end{aligned}
\end{equation}
so for instance 
\begin{multline}
 \mathfrak{r}_m(N;b_1,\dots ,b_m) \\ :=\sum_{n_1 \leqslant \frac{N}{b_1}, \dots , n_m \leqslant \frac{N}{b_m} } \mathbf{1}_{b_1 n_1 + \dots + b_mn_m=N} \mathbf{1}_{n_1, \dots ,n_m \text{ are prime}} ~= \sideset{}{^\flat}\sum_{n_1 \leqslant \frac{N}{b_1}, \dots , n_m \leqslant \frac{N}{b_m} } 1.
\end{multline}
Define 

\begin{equation}\label{5.1}
 \theta (n) =
\begin{cases}
	\log n, & \text{if $n$ is prime}\\
	0, & \text{otherwise}\\
  \end{cases}.
\end{equation}
\\
Hence, for some $A>0$ one has

\begin{align}\label{5.2}
&  \sum_{   \substack{ n_1, \dots,n_m \leqslant N \\ b_1 n_1+ \dots+ b_m n_m = N }} \Lambda (n_1) \cdots \Lambda (n_{m}) ~- 
 \sum_{   \substack{ n_1, \dots ,n_m \leqslant N \\ b_1 n_1+ \dots + b_m n_m = N }} \theta (n_1) \cdots \theta (n_{m})   \\
\leqslant& ~ \sum_{i=1}^m \sum_{   \substack{ n_1, \dots ,n_m \leqslant N \\ b_1 n_1+ \dots+ b_m n_m = N \\  n_i=p^k \text{ for } k>1 }} \Lambda (n_1) \cdots \Lambda (n_{m}) ~\leqslant ~  \log^{m} N \sum_{i=1}^m \sum_{   k_1 \leqslant \sqrt{N}, k_2, \dots ,k_{m-1}\leqslant N } 1 \\
\leqslant&~ N^{m-\frac{3}{2}}\log^m N ~ \leqslant  ~  \frac{N^{m-\frac{3}{2}+\delta m}\log^m N}{b_1 \cdots b_m}.
 \end{align}
 We are going to study the asymptotic behaviour of the function

\[ \mathfrak{r}_m(N;b_1,\dots ,b_m) := \sideset{}{^\flat}\sum_{  n_1\leqslant \frac{N}{b_1},\dots ,n_m \leqslant \frac{N}{b_m} } 1. \] 
Define also

\[ \widetilde{\mathcal{R}_m} (N;b_1, \dots ,b_m) :=  \sideset{}{^\flat}\sum_{  n_1\leqslant \frac{N}{b_1},\dots ,n_m \leqslant \frac{N}{b_m} }  \theta (n_1) \cdots \theta (n_{m}). \]

Let us prove the following result. 

\begin{lemma}\label{T5.1}
 Under the assumptions of Theorem \ref{T4.7} we have
\[ \mathfrak{r}_m (N;b_1,\dots ,b_m) = 
 \frac{1}{(m-1)!}\frac{1}{b_1\cdots b_m}\frac{N^{m-1}}{  \prod_{i=1}^m \log \frac{N}{b_i}} (\mathfrak{S}_{c_1,\dots ,c_m}(N)+o(1)). \]
\end{lemma}

\begin{proof}
Obviously $\widetilde{\mathcal{R}_m}=0$ if and only if $\mathfrak{r}_m=0$ and in such a case theorem follows trivially. Note that
\begin{equation}\label{5.3}
 \widetilde{\mathcal{R}_m} (N;b_1, \dots ,b_m) \leqslant  \mathfrak{r}_m (N;b_1, \dots ,b_m) \prod_{i=1}^m \log \frac{N}{b_i}.
\end{equation}
On the other hand, for every $\epsilon >0$ one has

\begin{multline}\label{5.4}
 \widetilde{ \mathcal{R}_m} (N;b_1, \dots ,b_m) \geqslant
 \sideset{}{^\flat}\sum_{ \left( \frac{N}{b_i} \right)^{1 -\epsilon} < n_i \leqslant \frac{N}{b_i}: \, 1\leqslant  i \leqslant m }   \theta (n_1) \cdots \theta (n_{m})  \\
\geqslant     (1- \epsilon)^m \left( \prod_{i=1}^m \log \frac{N}{b_i} \right) \sideset{}{^\flat}\sum_{ \left( \frac{N}{b_i} \right)^{1 -\epsilon} < n_i \leqslant \frac{N}{b_i}: \, 1\leqslant  i \leqslant m } 1.
\end{multline}   
From Lemma \ref{LA.2} the last sum in (\ref{5.4}) differs from $\mathfrak{r}_m$ by at most

\begin{multline}\label{szacowanie_kombi}
 \ll ~~ \sum_{j=1}^m \sum_{ \substack{ n_i \leqslant \frac{N}{b_i}: \, 1\leqslant  i \leqslant m \\ n_j \leqslant \left( \frac{N}{b_j} \right)^{1 -\epsilon}\\ b_1 n_1+ \dots + b_m n_m = N  } } 1 ~~ = ~~ \sum_{j=1}^m ~\sum_{n_j \leqslant \left( \frac{N}{b_j} \right)^{1 - \epsilon} } J_{b_1, \dots , b_{j-1}, b_{j+1} , \dots , b_m } (N - b_j n_j)    \\
\ll ~~ N^{m-2} \,
\sum_{j=1}^m ~\sum_{n_j \leqslant \left( \frac{N}{b_j} \right)^{1 - \epsilon} } \frac{1}{b_1 \cdots b_{j-1} b_{j+1} \cdots b_m}
~~ \ll ~~  \frac{N^{m-1-(1-\delta)\epsilon}}{b_1 \dots b_m}.
\end{multline} 

Now, we combine (\ref{5.3})--(\ref{szacowanie_kombi}) and we let $\epsilon$ very slowly approach $0$. In consequence
\begin{equation}
 \widetilde{ \mathcal{R}_m} (N;b_1, \dots ,b_m)  = (1 + o(1)) \left(  \mathfrak{r}_m (N;b_1, \dots ,b_m)\prod_{i=1}^m \log \frac{N}{b_i} \right) ~+~ O \left(  \frac{N^{m-1- \frac{1}{2}(1-\delta)\epsilon}}{b_1 \cdots b_m} \right).
\end{equation}
Combining this with Theorem \ref{T4.7} and (\ref{5.2}), we complete the proof. 
\end{proof}


\[ \]
\section{Cutting off}\label{CuttingOff}

Recall that $(c_1,\dots ,c_m)=1$. Let $N$ be large enough so that $Q:=\log^B N>\max \{c_1,\dots ,c_m\}$. In this section we finish the proof of Theorem \ref{MAIN} by making the key transition from primes to almost primes. Firstly, let us introduce the following simplifications. From this point forward we put 
\[ n_j := n_{j}^{(1)} \cdots n_{j}^{(r_j)}  ~~\,~~\,~~\text{and}~~\,~~\,~~ \eta_j := n_{j}^{(2)} \cdots n_{j}^{(r_j)}\]
 for all $j=1, \dots ,m$. We have to calculate the following expression:
\[ \sideset{}{^\natural}\sum_{  n_j^{(i)} \leqslant \frac{N}{c_j} \colon 1 \leqslant j \leqslant m, \, 1\leqslant i \leqslant r_j } 1. \]
Notice that for each choice of $n_j$ only one of the $n_j^{(i)}$ can be larger than $\sqrt{N}$. Therefore, our problem is equivalent (up to some overlapping cases when every of the $n_j^{(i)}$ is smaller than $\sqrt{N}$; the appropriate error term is calculated in (\ref{appro_error})) to calculating yet another expression:

\begin{equation}\label{6.1}
\sum_{ \substack{ n_1^{(1)}\leqslant \frac{N}{c_1}, \dots ,n_m^{(1)} \leqslant \frac{N}{c_m} \\ n_j^{(i)} \leqslant \sqrt{N}: \, 1 \leqslant j \leqslant m, \, 2\leqslant i \leqslant r_j  \\ \text{all of the } n_j^{(i)} \text{ being prime} }  }
 \mathbf{1}_{c_1 n_1 +\dots+c_m n_m =N}.
\end{equation}

We also adopt two new notions: 
\begin{equation}
\begin{aligned}
\sideset{}{^\natural}\sum_{ n_j^{(i)}: \, 1 \leqslant j \leqslant m, \, 1\leqslant i \leqslant r_j}  &:= 
\sum_{ \substack {  n_j^{(i)}: \, 1 \leqslant j \leqslant m, \, 1\leqslant i \leqslant r_j \\
\text{all of the } n_j^{(i)} \text{ being prime}\\ c_1 n_1 + \dots + c_mn_m=N  }} , \\
\sideset{}{'}\sum_{\text{some conditions}} ~~\,~~&:= \sum_{ \substack{ \text{some conditions} \\ \text{only prime indices} }}.
\end{aligned}
\end{equation}

The general strategy is to show that even after attaching some stronger conditions to the sum (\ref{6.1}), its asymptotic behaviour remains unchanged. Let $r:=\max \{ r_1,\dots ,r_m \}$. For $1 \leqslant i \leqslant m$ and $2 \leqslant j \leqslant r_i$ the range of indices $n_j^{(i)}$ from (\ref{6.1}) shall be cut from $[ 1 , \sqrt{N} ] \cap \mathbf{N}$ to $[ Q , N^{\gamma} ] \cap \mathbf{N}$ for some real $0 < \gamma < 1/2$. These further restrictions do not hurt the asymptotic behaviour of expression (\ref{6.1}), and simultaneously make it calculable via Lemma \ref{T5.1}. 

Let us focus on the details. Put $\gamma := \frac{1}{12rm(2m-1)}$. The expression (\ref{6.1}) can be decomposed as

\begin{equation}\label{6.2}
 {\sum}^{\dagger} + O \left( \sum_{\ell=1}^m \sum_{k=2}^{r_\ell} \left( {\sum}_1^{(\ell,k)} +  {\sum}_2^{(\ell,k)} \right) + {\sum}_3 \right), 
\end{equation}
where 
\begin{equation}
\begin{aligned}
{\sum}^{\dagger} &:=
 \sideset{}{^\natural}\sum_{ \substack{  n_1^{(1)}\leqslant \frac{N}{c_1}, \dots ,n_m^{(1)} \leqslant \frac{N}{c_m}  \\ Q < n_j^{(i)} \leqslant N^{\gamma}: \, 1 \leqslant j \leqslant m, \, 2\leqslant i \leqslant r_j \\  (\eta_1,\dots ,\eta_m)= 1}  } 1,
\\
  {\sum}_1^{(\ell,k)} &:=
\sideset{}{^\natural}\sum_{ \substack{ n_1^{(1)}\leqslant \frac{N}{c_1}, \dots ,n_m^{(1)} \leqslant \frac{N}{c_m} \\ n_j^{(i)} \leqslant \sqrt{N}: \, 1 \leqslant j \leqslant m, \, 2\leqslant i \leqslant r_j \\ n_{\ell}^{(k)} > N^{\gamma}   }}
1 ,  
\\
{\sum}_2^{(\ell,k)} &:=
\sideset{}{^\natural}\sum_{ \substack{ n_1^{(1)}\leqslant \frac{N}{c_1}, \dots ,n_m^{(1)} \leqslant \frac{N}{c_m} \\ n_j^{(i)} \leqslant N^{\gamma} : \, 1 \leqslant j \leqslant m, \, 2\leqslant i \leqslant r_j \\  n_{\ell}^{(k)} \leqslant Q  }}
 1,  
\\ 
{\sum}_3 &:=
\sideset{}{^\natural}\sum_{ \substack{ n_1^{(1)}\leqslant \frac{N}{c_1}, \dots ,n_m^{(1)} \leqslant \frac{N}{c_m} \\ Q < n_j^{(i)} \leqslant N^{\gamma}: \, 1 \leqslant j \leqslant m, \, 2\leqslant i \leqslant r_j \\  (\eta_1,\dots ,\eta_m) > 1}  }
 1.
\end{aligned}
\end{equation}
The assumption $\eta_1, \eta_2, \eta_3=1$ instantly gives
\[ {\sum}_3=0.  \]

\[ \]
\subsection{Estimating $ {\sum}^{\dagger}$} 

We can rewrite the studied sum in the following form:

\[  {\sum}^{\dagger} =  \sideset{}{'}\sum_{\substack{Q < n_j^{(i)} \leqslant N^{\gamma}: \\ 1 \leqslant j \leqslant m,~ 2\leqslant i \leqslant r_j  \\  (\eta_1, \dots ,\eta_m) = 1 } } \left(  \sideset{}{'}\sum_{ \substack{ n_1^{(1)}\leqslant \frac{N}{c_1 \eta_1}, \dots ,n_m^{(1)} \leqslant \frac{N}{c_m \eta_m}}  }
 \mathbf{1}_{c_1 \eta_1 n_{1}^{(1)}+\dots +c_m \eta_m n_{m}^{(1)} =N}  \right).
  \]
We have $\eta_j \leqslant N^{(r-1) \gamma }$ for all $j=1, \dots ,m$. For $N$ sufficiently large this gives $c_j \eta_j \leqslant N^{r \gamma}$. The sum in the parenthesis equals $\mathfrak{r}_m (N;c_1 \eta_1, \dots  , c_m \eta_m)$. Hence, the assumption $\eta_1,\eta_2,\eta_3=1$ and Lemma \ref{T5.1} imply that we can transform the expression above into:

\begin{multline}\label{6.3}
\frac{1}{(m-1)!} \frac{N^{m-1}}{c_1 \dots c_m} (\mathfrak{S}_{c_1, \dots ,c_m}(N)+o(1)) ~\times \\
 \sideset{}{'}\sum_{ Q < n_j^{(i)} \leqslant N^{\gamma}: \, 1 \leqslant j \leqslant m, \, 2\leqslant i \leqslant r_j} \frac{1}{\eta_1\dots \eta_m}\frac{1}{  \log \frac{N}{c_1 \eta_1} \dots \log \frac{N}{c_m \eta_m}} \mathbf{1}_{ (\eta_1, \dots ,\eta_m)=1} .
 \end{multline}
Given that $\mathbf{1}_{ (\eta_1, \dots ,\eta_m)=1} = 1$, we calculate
\begin{multline}\label{indukcja}
\prod_{j=1}^m \,  \sideset{}{'}\sum_{ Q < n_j^{(i)} \leqslant N^{\gamma}: \, 2\leqslant i \leqslant r_j}
\frac{1}{  \eta_j \log \frac{N}{c_j \eta_j}} ~=~ \\
\prod_{j=1}^m  \, \sideset{}{'}\sum_{ Q < n_{j}^{(r_j)} \leqslant N^{\gamma}} 
 \left( \dots \left( \sideset{}{'}\sum_{ Q < n_j^{(3)} \leqslant N^{\gamma}} \left(   \sideset{}{'}\sum_{ Q < n_j^{(2)} \leqslant N^{\gamma}}
\frac{1}{ n_j^{(2)}   \log \frac{N}{c_j \eta_j}} 
 \right) \frac{1}{n_j^{(3)}} \right) \dots \right)  \frac{1}{n_{j}^{(r_j)}}.
\end{multline}
From $ \frac{N}{c_j \eta_j}< N^{1 - r\gamma }$ for sufficiently large $N$ and from Lemma \ref{LB.1} applied $r_j-1$ times we conclude that the expression (\ref{indukcja}) equals

\begin{equation}\label{6.4}
 (1+o(1)) \prod_{j=1}^m  \frac{ (\log \log N)^{r_j -1}}{\log ( N/ c_j ) } = 
 (1+o(1)) \frac{ (\log \log N)^{r_1+\dots +r_m-m }}{ \log^m N }.
\end{equation}

Combining (\ref{6.3})--(\ref{6.4}) we arrive at

\begin{equation}\label{6.7}
 {\sum}^{\dagger} = \frac{1}{(m-1)!} \frac{N^{m-1}}{c_1\dots c_m} \frac{ (\log \log N)^{r_1+\dots +r_m-m }}{ \log^m N }
( \mathfrak{S}_{c_1,\dots ,c_m}(N) +o(1)).
\end{equation}

\[ \]
\subsection{Upper bounds on $ {\sum}_1^{(\ell,k)}$ and $ {\sum}_2^{(\ell,k)}$}\label{SS6.3} 

 We deal with these two sums in exactly the same way, so only the calculations for the first one are presented in detail. We can assume that $\ell \geqslant 4$ without loss of generality. 

The first sum can be rewritten in the following manner
\begin{multline}\label{hehe1}
 \sideset{}{^\natural}\sum_{ \substack{ n_1^{(1)}\leqslant \frac{N}{c_1}, \dots ,n_m^{(1)} \leqslant \frac{N}{c_m} \\ n_j^{(i)} \leqslant \sqrt{N}: \, 1 \leqslant j \leqslant m, \, 2\leqslant i \leqslant r_j \\ n_{\ell}^{(k)} > N^{\gamma}   }}1 =  \\
  \sideset{}{'}\sum_{ \substack{  n_4^{(1)}\leqslant \frac{N}{c_4}, \dots , n_m^{(1)}\leqslant \frac{N}{c_m}  \\
 n_j^{(i)} \leqslant \sqrt{N}: \, 1 \leqslant j \leqslant m, \, 2\leqslant i \leqslant r_j \\ \forall_j ~ c_j\eta_j \leqslant \frac{N}{2} \\ n_{\ell}^{(k)} > N^{\gamma}   }}
\sideset{}{'}\sum_{ n_1^{(1)}\leqslant \frac{N}{c_1},  n_2^{(1)}\leqslant \frac{N}{c_2} , n_3^{(1)}\leqslant \frac{N}{c_m} }
  \mathbf{1}_{{c_1 \eta_1 n_1^{(1)} }+ \dots + c_3 \eta_3 n_3^{(1)} =N -
  c_4 \eta_4 n_4^{(1)} - \dots- c_m \eta_m n_m^{(1)} }.
 \end{multline}
From  $\eta_1=\eta_2=\eta_3 =1$ and  Lemma \ref{L6.5} the inner sum from (\ref{hehe1}) equals
\[ \mathfrak{r}_m( N -
  c_4 \eta_4 n_4^{(1)} -\dots - c_m \eta_m n_m^{(1)}
   ; c_1 , c_2 , c_3 ) \ll \frac{N^2}{ \log^3 N }. \]
Thus, we obtain
\begin{multline}
  {\sum}_1^{(\ell,k)} \ll \frac{N^2}{ \log^3 N } \sideset{}{'}\sum_{ \substack{  n_4^{(1)}\leqslant \frac{N}{c_4}, \dots , n_m^{(1)}\leqslant \frac{N}{c_m} \\
n_j^{(i)} \leqslant \sqrt{N} \colon 4 \leqslant j \leqslant m, \, 2\leqslant i \leqslant r_j \\ \forall_j ~ n_j \leqslant N \\ n_{\ell}^{(k)} > N^{\gamma}   }} 
 1 \ll  
   \\
    \frac{N^2}{ \log^3 N } 
  \sum_{ \substack{  n_j \leqslant N:\\ 4 \leqslant j \leqslant m, j \not= \ell   }} 
 \left(  \prod_{\substack{ l=4 \\ l \not= \ell}}^m \mathbf{1}_{\Omega (n_l) = r_l} \right)
   \sideset{}{'}\sum_{\substack{ n_\ell^{(1)}, \dots , n_\ell^{(r_{\ell})} \leqslant N \\ n_\ell \leqslant N \\  N^{\gamma} < n_{\ell}^{(k)} \leqslant \sqrt{N} }} 1.
\end{multline}
From Theorem \ref{Landau} we have

\begin{multline} \sideset{}{'}\sum_{\substack{ n_\ell^{(1)}, \dots , n_\ell^{(r_{\ell})} \leqslant N \\ n_\ell \leqslant N \\  N^{\gamma} < n_{\ell}^{(k)} \leqslant \sqrt{N} }} 1 ~\ll~ \sideset{}{'}\sum_{  N^{\gamma} < n \leqslant \sqrt{N} }
~ \sum_{ h \leqslant \frac{N}{  n }  } \mathbf{1}_{\Omega (h) = r_\ell -1} 
\\
\ll ~ \frac{ N (\log \log N )^{r_\ell -2}}{\log N} \sideset{}{'}\sum_{  N^{\gamma} < n \leqslant \sqrt{N} }
\frac{  1 }{ n  } ~\ll~  \frac{ N (\log \log N )^{r_\ell -2}}{\log N},
\end{multline}
and for $ j \not= \ell$ we have
\[  \sum_{  n_j\leqslant N } 
  \mathbf{1}_{\Omega (n_j) = r_j} \ll \frac{N (\log \log N)^{r_j -1}}{\log N}, \]
  which gives
\begin{equation}
  {\sum}_1^{(\ell,k)} \ll  \frac{ N^{m-1} (\log \log N)^{r_1+\dots +r_m -m-1} }{\log^m N}. 
  \end{equation}
Analogously, we can obtain the following upper bound for the second sum
\[  {\sum}_2^{(\ell,k)} \ll \frac{ N^{m-1} (\log \log N)^{r_1+\dots +r_m -m-1}(\log \log \log N) }{\log^m N}. \]
The $\log \log \log N$ term appears because this time $n_\ell^{(k)}$ is an index supported on $[1,Q] \cap \mathbf{Z}$ instead of $(N^\gamma ,N^{1/2}] \cap \mathbf{Z}$ like in the first case; we easily get

\[  \sideset{}{'}\sum_{ n_{\ell}^{(k)} \leqslant Q }
\frac{  1 }{ n_{\ell}^{(k)}  } \ll \log \log Q \ll \log \log \log N.\]
\[ \]

\subsection{Proof of Theorem \ref{MAIN}} 

Now, we can summarize the discussion from this section. Under the assumptions of Theorem \ref{MAIN} we have
 
\begin{equation}\label{6.18}
\sideset{}{^\natural}\sum_{ \substack{ n_1^{(1)}\leqslant \frac{N}{c_1}, \dots ,n_m^{(1)} \leqslant \frac{N}{c_m} \\ n_j^{(i)} \leqslant \sqrt{N}: \, 1 \leqslant j \leqslant m, \, 2\leqslant i \leqslant r_j }  }1 ~=~  \frac{1}{(m-1)!} \frac{N^{m-1}}{c_1\dots c_m} \frac{ (\log \log N)^{r_1+\dots +r_m-m }}{ \log^m N }
( \mathfrak{S}_{c_1,\dots ,c_m}(N) +o(1)).
\end{equation}
We wish to discard the $n_j^{(i)} \leqslant \sqrt{N}$ restriction from the sum above. Notice that every $\eta_j$ has at most one term greater than $\sqrt{N}$. Hence,
\begin{multline}\label{6.19}
\sideset{}{^\natural}\sum_{  n_j^{(i)} \leqslant \frac{N}{c_j} \colon 1 \leqslant j \leqslant m, \, 1\leqslant i \leqslant r_j }
 1 ~=~  r_1 \cdots  r_m \sideset{}{^\natural}\sum_{ \substack{ n_1^{(1)}\leqslant \frac{N}{c_1}, \dots  ,n_m^{(1)} \leqslant \frac{N}{c_m} \\ n_j^{(i)} \leqslant \sqrt{N}: \, 1 \leqslant j \leqslant m, \, 2\leqslant i \leqslant r_j }  }
1   \\
 + ~ O \left(   \sum_{K=1}^m ~
\sideset{}{^\natural}\sum_{ \substack{ n_1^{(1)}\leqslant \frac{N}{c_1}, \dots ,n_K^{(1)} \leqslant \sqrt{N},\dots  ,n_m^{(1)} \leqslant \frac{N}{c_m} \\ n_j^{(i)} \leqslant \sqrt{N}: \, 1 \leqslant j \leqslant m, \, 2\leqslant i \leqslant r_j }  } 1
  \right). 
 \end{multline}
We assumed that $\eta_1=\eta_2=\eta_3=1$. The main term in the expression above equals
\begin{equation}\label{6.20}
 \frac{r_1 \cdots r_m}{(m-1)!} \frac{N^{m-1}}{c_1\dots c_m} \frac{ (\log \log N)^{r_1+\dots +r_m-m }}{ \log^m N }
( \mathfrak{S}_{c_1,\dots ,c_m}(N) +o(1)).
\end{equation}
Let us study the error term from (\ref{6.19}). We have
\begin{multline}\label{appro_error} \sideset{}{^\natural}\sum_{ \substack{ n_1^{(1)}\leqslant \frac{N}{c_1}, \dots , ,n_K^{(1)} \leqslant \sqrt{N},\dots  ,n_m^{(1)} \leqslant \frac{N}{c_m} \\ n_j^{(i)} \leqslant \sqrt{N}: \, 1 \leqslant j \leqslant m, \, 2\leqslant i \leqslant r_j }  } 1
\\
 = \sideset{}{^\natural}\sum_{ \substack{ n_1^{(1)}\leqslant \frac{N}{c_1}, \dots , ,n_K^{(1)} \leqslant \sqrt{N},\dots  ,n_m^{(1)} \leqslant \frac{N}{c_m} \\ Q < n_j^{(i)} \leqslant N^\gamma: \, 1 \leqslant j \leqslant m, \, 2\leqslant i \leqslant r_j }  }
1 ~+~  O \left( \sum_{\ell=1}^m \sum_{k=2}^{r_\ell} \left( {\sum}_1^{(\ell,k)} +  {\sum}_2^{(\ell,k)} \right) \right).
 \end{multline}
The 'big O' term has order at most as large as the error term from (\ref{6.20}). The new main term (which is essentially only the error term in (\ref{6.19})) can be bounded as $\leqslant N^{m-\frac{3}{2}+ rm\gamma }$.

Combining the results from this subsection we can state that
\begin{multline}\label{6.21}
 \sideset{}{^\natural}\sum_{  n_j^{(i)} \leqslant \frac{N}{c_j} \colon 1 \leqslant j \leqslant m, \, 1\leqslant i \leqslant r_j } 1 
  \\
  = \frac{r_1 \cdots r_m}{(m-1)!} \frac{N^{m-1}}{c_1 \cdots c_m} \frac{ (\log \log N)^{r_1+\dots +r_m-m }}{ \log^m N }
( \mathfrak{S}_{c_1,\dots ,c_m}(N) +o(1)).
\end{multline}

By Lemma \ref{L6.5} and Theorem \ref{Landau} we can repeat the trick from the previous subsection to obtain the following upper bound for some $1 \leqslant \ell \leqslant m$ and $1 \leqslant k_1 , k_2 \leqslant r_\ell$ for which $k_1 \not= k_2$: 
\begin{multline}
\sideset{}{^\natural}\sum_{  \substack{ n_j^{(i)} \leqslant \frac{N}{c_j} \colon 1 \leqslant j \leqslant m, \, 1\leqslant i \leqslant r_j \\ n_\ell^{(k_1)}=n_\ell^{(k_2)}  }} 1 ~\ll~
\frac{N^2}{\log^3 N} \sideset{}{'}\sum_{\substack{  n_j^{(i)} \leqslant \frac{N}{c_j} \colon 1 \leqslant j \leqslant m, \, 1\leqslant i \leqslant r_j \\ \forall_j ~~  n_j \leqslant N \\ n_\ell^{(k_1)}=n_\ell^{(k_2)} } } 1
\\
 \ll ~\frac{ N^{m-1} (\log \log N)^{r_1+\dots +r_m -m-1} }{\log^m N}.
\end{multline}
Therefore, we have the following identity
\begin{multline}\label{6.22}
\sideset{}{^\natural}\sum_{  n_j^{(i)} \leqslant \frac{N}{c_j} \colon 1 \leqslant j \leqslant m, \, 1\leqslant i \leqslant r_j } 1 
~=~
 r_1! \cdots r_m!
\sum_{   \substack{  n_1, \dots ,n_m \leqslant N \\  c_1 n_1+ \dots + c_m n_m = N }} \prod_{i=1}^m \mathbf{1}_{\Omega(n_i) = r_i} 
\\
+~ O \left( \frac{ N^{m-1} (\log \log N)^{r_1+\dots +r_m -m-1} }{\log^m N} \right) .
\end{multline}
This finally closes the proof of Theorem \ref{MAIN}.


\[ \]
\appendix
\section{}

\begin{lemma}\label{L6.5}   Let $b_1, \dots ,b_m \leqslant N$ be positive integers such that $b_1, b_2, b_3$ are fixed. Then,
\[ \mathfrak{r}_m(N;b_1,\dots ,b_m) \ll \frac{ N^{m-1} }{b_1 \dots b_m \prod_{j=1}^m \log (N/b_j)}. \]
\end{lemma}
\begin{proof}

From Theorem \ref{T4.7} we have
\[ \mathfrak{r}_3(N;b_1,b_2,b_3) \ll  \frac{ N^2}{ \log^3 N} \]
for sufficiently large integer $N$.

We can present $\mathfrak{r}_m(N;b_1, \dots ,b_m)$ as
\[ 
\sideset{}{'}\sum_{ n_4 \leqslant \frac{N}{b_4}, \dots , n_m \leqslant \frac{N}{b_m} } \left( 
\sideset{}{'}\sum_{n_1 \leqslant \frac{N}{b_1}, n_2 \leqslant \frac{N}{b_2}, n_3 \leqslant \frac{N}{b_3} }
 \mathbf{1}_{b_1 n_1 + b_2n_2 + b_3n_3 =N - b_4n_4-\dots -b_mn_m} 
\right) .
\] 
The term inside the bracket equals $\mathfrak{r}_3(N-b_4n_4- \dots -b_mn_m; b_1,b_2,b_3)$, so one gets 

\begin{multline*} \mathfrak{r}_m(N;b_1, \dots ,b_m) \ll \frac{N^2 }{b_1b_2b_3 \log^3 N} 
 \sideset{}{'}\sum_{n_4 \leqslant \frac{N}{b_4}, \dots , n_m \leqslant \frac{N}{b_m} }   1  \\
\ll \frac{N^2}{b_1b_2b_3 \log^3 N} \prod_{j=4}^m \frac{N}{b_j \log (N/b_j)}
 \ll \frac{ N^{m-1}  }{b_1 \dots b_m \prod_{j=1}^m \log (N/b_j) }. 
 \end{multline*}
\end{proof}

\begin{lemma}\label{tau_sq} For $x \geqslant 2$ we have 
\[ \sum_{n \leqslant x} \tau (n)^2 \ll x \log^3 x. \]
\end{lemma}

\begin{proof} 
See \cite{Raman}.
\end{proof}

\begin{lemma}\label{LB.1}
Let $x,y \in \mathbf{R}$. Then, for any $\delta \in (0,1)$ we have
\[  \sum_{p \leqslant x^\delta}  \frac{1}{p \log \frac{x}{p}} = (1+o_\delta (1)) \frac{\log \log x}{\log x}.\]
\end{lemma}

\begin{proof}
Apply summation by parts and the prime number theorem.
\end{proof}

\begin{theorem}[Landau]\label{Landau}
 For $k \geqslant 1$ we have
\[ \sum_{n \leqslant x} \mathbf{1}_{\Omega(n) = k} = (1+o(1)) \frac{x (\log \log x)^{k-1}}{(k-1)! \log x}. \]
\end{theorem}

\begin{proof}
See \cite{Y}.
\end{proof}

\begin{lemma}\label{L3.2}
Fix positive integers $c_1, \dots, c_m$ satisfying $(c_1,\dots ,c_m)=1$. It is true that $ \mathfrak{S}_{c_1,\dots ,c_m}(N) \not= 0$ if and only if $c_1+\dots +c_m+N \equiv 0 \mod 2$ and
\[ (N,c_2,\dots ,c_m)=\dots =(c_1,\dots ,c_{i-1},N,c_{i+1},\dots c_m)=\dots =(c_1,\dots ,c_{m-1},N)=1.\]
Moreover, $\mathfrak{S}_{c_1,\dots ,c_m}(N) \gg 1$ for every $N$ such that $ \mathfrak{S}_{c_1,\dots ,c_m}(N) \not= 0$.
\end{lemma}

\begin{proof}
Let us prove the 'moreover' part first. We may assume that no factor of $ \mathfrak{S}_{c_1,\dots ,c_m}(N)$ equals 0. Hence, if $p|N$, then at least 2 of the $c_i$ are not divisible by $p$, because otherwise either the term corresponding to $p$ would vanish or $p|(c_1,\dots, c_m)$. Therefore,
\begin{align*}
1 + (p-1) \frac{\mu \left( \frac{p}{(c_1,p)} \right) \cdots {\mu \left( \frac{p}{(c_m,p)} \right)} }{\varphi \left( \frac{p}{(c_1,p)} \right) \cdots {\varphi \left( \frac{p}{(c_m,p)} \right)}} &~\geqslant~ 1 - \frac{1}{p-1}, \\
1 - \frac{\mu \left( \frac{p}{(c_1,p)} \right) \cdots {\mu \left( \frac{p}{(c_m,p)} \right)} }{\varphi \left( \frac{p}{(c_1,p)} \right) \cdots {\varphi \left( \frac{p}{(c_m,p)} \right)}} &~\geqslant~ 1 - \frac{1}{p-1}. 
\end{align*}
The second inequality works because we are always guaranteed to have at least one of the $c_i$ that is not divisible by $p$. 
It is also true that there exist at most finitely many primes dividing $\prod_{i=1}^m c_i$. Let $z$ be any fixed real number greater than all of them.
Hence,
\begin{equation}
 \mathfrak{S}_{c_1,\dots ,c_m}(N) \geqslant 2\prod_{ \substack{2< p\leqslant z}}\left( 1 - \frac{1}{p-1} \right)  \prod_{p>z} \left( 1 - \frac{1}{(p-1)^{m-1}} \right) \gg 1.
\end{equation}

Let us consider the '$\Longrightarrow$' case. From the previous part we already know that $\mathfrak{S}_{c_1,\dots ,c_m}(N)$ equals 0 if and only if one of its factors vanishes. For any prime $p \geqslant 3$ the factor
\[ 1 + (p-1) \frac{\mu \left( \frac{p}{(c_1,p)} \right) \cdots {\mu \left( \frac{p}{(c_m,p)} \right)} }{\varphi \left( \frac{p}{(c_1,p)} \right) \cdots {\varphi \left( \frac{p}{(c_m,p)} \right)}}, \]
related to the $p|N$ case, vanishes if and only if every of the $c_i$ is divisible by $p$ except exactly one of them. The second factor actually never vanishes except possibly when $p=2$. In the $p=2$ case we note that factors related to $N$ being even and $N$ being odd simplify to
\[ 1 + (-1)^{ \# \{1 \leqslant i \leqslant m \colon 2 \nmid c_i \} }   \mbox{ \ \ \ \ \ and \ \ \ \ \ } 1 - (-1)^{ \# \{1 \leqslant i \leqslant m \colon 2 \nmid c_i \} }\]
respectively. It means that whenever $N$ has different parity than the sum $c_1+\dots+c_m$, one of these has to vanish. 

The '$\Longleftarrow$' case is now straightforward, because if $\mathfrak{S}_{c_1,\dots ,c_m}(N)=0$, then one of its factors vanishes. Consequently, either parities of $c_1 + \dots + c_m$ and $N$ do not match or there exists some prime $p\geqslant 3$ dividing $N$ and each of the $c_i$ except exactly one. 
\end{proof}

\bibliographystyle{amsplain}

\end{document}